\documentclass{amsart}

\usepackage{graphicx}
\usepackage[utf8]{inputenc}
\usepackage[shortlabels]{enumitem}
\usepackage{tikz-cd}

\numberwithin{equation}{section}

\newtheorem{theorem}{Theorem}[section]
\newtheorem{lemma}[theorem]{Lemma}

\newtheorem{remark}[theorem]{Remark}

\newcommand{\R}{\mathbb{R}}

\begin{document}

\title[Free boundary minimal hypersurfaces in the Schwarzschild space]{Proper free-boundary minimal hypersurfaces with a rotational symmetry in the Schwarzschild space}

\author[Barbosa and Moya]{Ezequiel Barbosa and David Moya*}\thanks{*Research partially supported by a MINECO/FEDER grant no. MTM2017-89677-P, MICINN grant PID2020-117868GB-I00,  AEI/10.13039/501100011033/, IMAG–Maria de Maeztu grant CEX2020-001105-M and A-FQM-139-UGR18}

\address{Departamento de Geometría y Topología, Facultad de Ciencias. Universidad de Granada, 18071-Granada-Spain}

\email{ezequiel@mat.ufmg.br}
\email{dmoya@ugr.es}

\begin{abstract}
In this work we present a new family of properly embedded free boundary minimal hypersurfaces of revolution with circular boundaries in the horizon of the $n$-dimensional Schwarzschild space, $n\geq3$. In particular, we answer a question proposed by O. Chodosh and D. Ketover \cite{CK} on the existence of non-totally geodesic minimal surfaces in the 3-dimensional Schwarzschild space. 

\end{abstract}

\maketitle

\section{Introduction}

Let $n$ be an integer, $n\geq 3$, and $k\in \mathbb{R}$, $k\geq1$, $n\neq 2k$. For each real number $m>0$, we consider the $n$-dimensional domain
\[
M^{n}:=\left\{\vec{x}\in\mathbb{R}^{n}\,;\,\, |\vec{x}|\geq R_{0} \right\}\,,
\]
$R_{0} := \left(\frac{m}{2}\right)^{\frac{k}{n-2k}}$, endowed with the Riemannian metric
\[
g_{Sch(k)}=\left( 1+\frac{m}{2|\vec{x}|^{\frac{n}{k}-2}} \right)^{\frac{4k}{n-2k}}\,\langle,\rangle\, ,
\]
where $|\vec{x}|^2=|(x_1,\dots,x_n)|^2=\sum\limits_{i=1}^{n}x_i^2$ and $\langle,\rangle$ denotes the Euclidean metric. Those metrics are the general Schwarzschild metric. When $k=1$ we recover the very well known Schwarzschild space. We write $M^n_k$ to denote the domain $M^n$ endowed with the Riemannian metric $g_{Sch(k)}$. These general Schwarzschild spaces are very important examples of asymptotically flat manifolds. In fact, the simplest solution to the Einstein equations that contains an asymptotically flat manifold as an initial data set is the Lorentz-Minkowski space where the Euclidean space is the initial data set. The second one, found in 1916 by K. Schwarzschild \cite{Sch16}, is the Schwarzschild vacuum which is a solution to the Einstein field equations that describes the gravitational field outside a spherical mass; which initial data set is known as the Riemannian Schwarzschild manifold. The significant difference of this particular solution is the existence of a "singularity" and it is surrounding a spherical boundary, called the event horizon.The Riemannian Schwarzschild manifold is a fundamental example in the study of static black holes. Recently, other interesting asymptotically flat manifolds with event horizons (more general Schwarzschild type metrics), which play an important role as the Schwarzschild space in Einstein gravity, have been studied in the Einstein-Gauss-Bonnet theory and the pure Lovelock theory (see, for instance, \cite{9}, \cite{15}, \cite{20A}, \cite{33}, \cite{38}).

 A key ingredient to study physical properties of these manifolds are the minimal hypersurfaces contained in them; as the proof of the Positive Mass Theorem, given by R. Schoen and Yau \cite{RSchSYau81,RSchSYau17}, has shown.  The aim of this work is to study the existence of certain type of minimal hypersurfaces with boundary meeting orthogonally the horizon of a Riemannian general Schwarzschild space.

We can check that the boundary of $M_k^n$, the horizon $S_{0}=\{\vec{x}\in\mathbb{R}^{n}\,;\,\, |\vec{x}|=R_{0} \}$, is a closed totally geodesic hypersurface in $M^n_k$. Also, we consider the totally geodesic Euclidean hypersurfaces which are the rotations of the coordinate hyperplane $\Sigma _0$ through the origin minus the open ball $B_{0}:=B( R_{0})$ of radius $R_0$ centered at the origin: 
\begin{equation}\label{Sigma}
\Sigma_0=\{\vec{x}\in\mathbb{R}^{n}\,;\,\, x_{n}= 0 \,,\,\,|\vec{x}|\geq R_{0} \}.
\end{equation}

Let $u:\R \rightarrow \R$ be a differentiable function and $h:\R^n\rightarrow \R$ be given by $h(\vec{x})=u(|\vec{x}|^2)$ and consider the conformal change $\overline{g}=e^{2h}\langle,\rangle$. Denote by $k_i$ and $\overline{k}_i$ the principal curvatures of a hypersurface $\Sigma$ with respect to the Euclidean metric $\langle,\rangle$ and the metric $\overline{g}$, respectively. We have that $k_i$ and $\overline{k}_i$ are related as
\[
\overline{k}_i=e^{-h}\left(k_i-2u'(|\vec{x}|^2)\langle \vec{x},N\rangle\right)\qquad i=1,\dots,n-1,
\]
where $N$ stands for the normal vector at the point $\vec{x}\in \Sigma$. As the function $\langle \vec{x},N\rangle$ vanishes on $\Sigma_0$, we have that hypersurface $\Sigma_0$ is a properly embedded free boundary totally geodesic, in particular minimal, hypersurface in $M^n_k$, i.e., the boundary $\partial \Sigma_0$ coincides with $\Sigma_0 \cap \partial M$, its mean-curvature with respect to $g_{Sch}$ vanishes and the boundary $\partial \Sigma_0$ meets $\partial M$ orthogonally.
 
For the 3-dimensional Riemannnian Schwarzschild space $M^3_1$, O. Chodosh and D. Ketover \cite{CK} proposed to study the Morse index of the annuli surface $\Sigma_0$, i.e., the maximum number of directions, tangential along $\partial M$, in whose the surface can be deformed in such a way that its area decreases, and also to investigate the existence of unbounded non-totally geodesic free boundary minimal surfaces (see page 5 in \cite{CK}): "are these annuli and the horizon the only embedded
minimal surfaces in Schwarzschild?" It follows from the work of A. Carlotto \cite{C16} that the Morse index of $\Sigma_0$ can not be zero. Very recently, R. Montezuma \cite{RaMo} computed the Morse index of $\Sigma_0$: it was proved that the Morse index of $\Sigma_0$ is one. In this work we consider the existence question and we prove that the annuli $\Sigma_0$ is not the only properly embedded free boundary minimal surface in the 3-dimensional general Schwarzschild space.

\begin{theorem}\label{MAIN}
For each $t_0\in (0,R_0)$ there exists a proper non totally geodesic rotationally symmetric free boundary minimal surface $\Sigma_{t_0}$ in the Schwarszchild space $M^3_k$, that intersects $\partial M^3_k$ at height $x_3=t_0$. Moreover, the height of $\Sigma_{t_0}$ with respect to the $x_3=0$ plane is unbounded.  
\end{theorem}

It is worth mention that it is not clear that one can perturb an unbounded minimal surface in the Euclidean space in order to obtain a minimal surface in the 3-dimensional Schwarzschild spaces $M^3_k$. An obstruction to a particular such deformation, considering catenoids, was demonstrated in \cite{CarMon}. Therefore, we apply a different approach. Our technique is more related to the one applied in the work of S. J. Kleene and N. M. Moller \cite{KleeMoller}.

It would be interesting to compute the Morse index of these free boundary minimal surfaces in the general Schwarzschild space $M^3_k$. For $k=1$, we know that this surfaces are not stable (see \cite{C16}). In section 3, we discuss on the Morse index of $\Sigma_0$ in $M_k^3$ and the total curvature of $\Sigma_{t_0}\subset M^3_k$.

For the $n$-dimensional general Riemannnian Schwarzschild space, $n\geq4$, the existence and the Morse index situation change. We give a brief description here. Let 
\[
C_{\Gamma}:=\{ \lambda y \, ; \,y \in \Gamma^{n-2}, \,\lambda\in (0,\infty)\}
\]
be a cone in $\mathbb{R}^{n}$, $n\geq 4$, with vertex at the origin, here $\Gamma^{n-2}$ is an embedded closed orientable minimal hypersurface in $\mathbb{S}^{n-1}$. We refer to $C_{\Gamma}$ as the {\it minimal cone over ${\Gamma}$}. Finally, consider $\Sigma_{\Gamma}=\{\vec{x}\in C_{\Gamma}\,;\,|\vec{x}|\geq R_{0} \}$. This  hypersurface is a properly embedded free boundary minimal hypersurface in $M^n$. Note that $\Sigma_0=\Sigma_{\Gamma_0}$ when $\Gamma_0$ is the equator in $\mathbb{S}^{n-1}$ given by the intersection $\mathbb{S}^{n-1}\cap H_0$, where $H_0=\{\vec{x}\in\mathbb{R}^{n}\,;\,\, x_{n}= c_0 \,\}$ is a hyperplane passing through the center of $\mathbb{S}^{n-1}$. In \cite{BarEsp}, the first named author and J. Espinar studied the Morse index of those minimal cones and proved that the Morse index is zero, for the totally geodesic hypersurface $\Sigma_0$ and, for the non-totally geodesic minimal cones $\Sigma_{\Gamma}$, the Morse index is either zero or infinite, depending on the dimension $n\geq 4$. In this same work, it was proposed a question on the existence or not of properly embedded free boundary minimal hypersurfaces in the Riemannnian Schwarzschild space, $n\geq 4$, with Morse index one (or, at least, finite non-zero Morse index). Here we deal with the question of existence of properly embedded free boundary minimal hypersurfaces non congruent to the cones $\Sigma_{\Gamma}$.

\begin{theorem}\label{newMAIN}
Let $M^n_k$ be a general Schwarzschild space with $n\geq4$. For each $t_0\in (0,R_0)$ there exists a proper rotationally symmetric free boundary minimal hypersurface $\Sigma_{t_0}$ in $M^n_k$, that intersects $\partial M^n$ at height $x_n=t_0$.

Moreover, $\Sigma_{t_0}$ is bounded above by the hyperplane of height $h_0$, where
\[
h_0=t_0+\int_{\sqrt{R_0^2-t_0^2}}^{+\infty}\frac{d\mu}{\sqrt{\left(\frac{R_0}{t_0}\right)^2\left(\frac{\mu}{\sqrt{R_0^2-t_0^2}}\right)^{2(n-2)}-1}}\,,
\]
and $\Sigma_{t_0}$ goes to $\Sigma_0$ as $t_0\rightarrow0$.
\end{theorem}

In \cite{CarMon}, A. Carlotto and A. Mondino showed the existence of catenoidal minimal hypersurfaces in asymptotically Schwarzschildean spaces, in this case of higher dimensions, by perturbing the Euclidean catenoids. In our case, our minimal hypersurfaces are rotationally symmetric catenoidal type minimal hypersurfaces.  

In a slightly general situation, if $u$ is a differentiable function with  $u'(|\vec{x}|^2)<0$ and $-u'(|\vec{x}|^{2})\leq\frac{c}{|\vec{x}|^{\alpha}}$ for some positive constants $c$ and $\alpha$, $\alpha\geq2$, consider a conformal change $\overline{g}=e^{2h}\langle,\rangle$, where $h=u(|\vec{x}|^2)$, on the $n$-dimensional domain, $n\geq3$,
\[
M^{n}:=\left\{\vec{x}\in\mathbb{R}^{n}\,;\,\, |\vec{x}|\geq r_{0} \right\}\,,
\]
$r_0>0$. In this situation, we obtain the following existence result.

\begin{theorem}\label{newMAIN2}
For each $t_0\in (0,r_0)$ there exists a proper rotationally symmetric free boundary minimal hypersurface $\Sigma_{t_0}$ in $(M^n,\overline{g})$, that intersects $\partial M^n$ at height $x_n=t_0$. Moreover:

\begin{itemize}
\item[1.] if $n=3$, the height of $\Sigma_{t_0}$ with respect to the $x_3=0$ plane is unbounded.  

\item[2.] If $n\geq4$, $\Sigma_{t_0}$ is bounded above by the hyperplane of height $h_0$, where
\[
h_0=t_0+\int_{\sqrt{r_0^2-t_0^2}}^{+\infty}\frac{d\mu}{\sqrt{\left(\frac{r_0}{t_0}\right)^2\left(\frac{\mu}{\sqrt{r_0^2-t_0^2}}\right)^{2(n-2)}-1}}\,,
\]
and $\Sigma_{t_0}$ goes to $\Sigma_0$ as $t_0\rightarrow0$.
\end{itemize}
\end{theorem}

For instance, consider the cylinder $\mathbb{S}^2\times \mathbb{R}$ with the usual product metric. It is isometric to the space $(\mathbb{R}^3\setminus\{ 0\}, \overline{g})$, where $\overline{g}=e^{2h}\langle,\rangle$ and $h(|\vec{x}|)=u(|\vec{x}|^2)=\frac{1}{2}\ln\left( \frac{1}{|\vec{x}|^2}\right)$.  The isometry $F:\mathbb{R}^3\setminus\{ 0\}\rightarrow \mathbb{S}^2\times \mathbb{R}$ is given by 
\[
F(\vec{x})=(|\vec{x}|^{-1}\vec{x},\ln (|\vec{x}|))\,.
\]
Note that $u'(|\vec{x}|^2)<0$ and $-u'(|\vec{x}|^{2})\leq\frac{1}{|\vec{x}|^{2}}$. We can see that, for all $r_0>0$, the sphere $S_{r_0}=\{\vec{x}\in\mathbb{R}^{3}\,;\,\, |\vec{x}|=r_{0} \}$ is a closed totally geodesic hypersurface in $(M^3,\overline{g})$. Hence, for every $t_0\in (0,r_0)$ there exists a proper rotationally symmetric free boundary minimal hypersurface $\Sigma_{t_0}$ in $(M^3,\overline{g})$, where $M^3=\left\{\vec{x}\in\mathbb{R}^{3}\,;\,\, |\vec{x}|\geq r_{0} \right\}\,$. Note that the surface $\Sigma_{t_0}$ is not totally geodesic. Then it is not a surface as $\gamma\times [r_0,+\infty)$ where $\gamma$ is a geodesic in $\mathbb{S}^2$. 

Also, consider a real number $\beta\in ]0, \frac{1}{2}]$ and $h(r)=-\beta\ln(1 + r^2)$. Then the metric $\overline{g}=e^{2h}\langle,\rangle$ is a complete metric on $\mathbb{R}^n$. Note that $u'(|\vec{x}|^2)=-\frac{\beta}{1+|\vec{x}|^2}<0$ and $-u'(|\vec{x}|^{2})\leq\frac{\beta}{|\vec{x}|^{2}}$. Hence, for every $t_0\in (0,r_0)$ there exists a proper rotationally symmetric free boundary minimal hypersurface $\Sigma_{t_0}$ in $(M^n,\overline{g})$, where $M^n=\left\{\vec{x}\in\mathbb{R}^{n}\,;\,\, |\vec{x}|\geq r_{0} \right\}\,$. 

The same technique we use in the Schwarzschild space could be applied to prove that if we consider a hyperplane in $\mathbb{R}^n$ then there exists rotationally symmetric minimal hypersurfaces, with respect to the metric $\overline{g}$, contained in a halfspace, embedded and with boundary on the hyperplane that defines the halfspace in $\mathbb{R}^n$.

\section{Rotationally Symmetric Free-Boundary Minimal Hypersurfaces
}
Let $x:I\rightarrow \R$ be a smooth function with $x(t)>0$, $\forall t\in I$, where $I\subset \R$ is an open interval, such that the curve
\[
\beta(t)=(x(t),0,\dots,0,t)
\]
is a graph which is not necessarily arclenght parametrized. Consider the hypersurface $\Sigma\subset (\R^n,\langle,\rangle)$ obtained by revolution of the curve $\beta$ around the $x_n$-axis, whose parametrization $\vec{x}=X:[0,+\infty)\times \mathbb{S}^{n-2}\rightarrow \R^n$ is given by
\begin{equation}\label{1}
X(t,\theta)=(x(t)\theta,t),
\end{equation}
where $\theta$ denotes a point in $\mathbb{S}^{n-2}$. Consider in $\Sigma$ the orientation:
\[
N=\frac{1}{\sqrt{x'(t)^2+1}}(-\theta,x'(t)).
\]
The principal curvatures of $\Sigma$ at a point $(t,\theta)$ are given by
\begin{equation}\label{curvatures}
k_1=-\frac{x''(t)}{\sqrt{(1+x'(t)^2)^3}}, \, k_2=k_3=\dots=k_{n-1}=\frac{1}{x(t)\sqrt{1+x'(t)^2}}.
\end{equation}

Let $u:\R \rightarrow \R$ be a differentiable function and $h:\R^n\rightarrow \R$ be given by $h(x)=u(|x|^2)$, under conformal change $\overline{g}=e^{2h}\langle,\rangle$, the principal curvatures $\overline{k}_i$ become
\begin{equation}\label{2}
\overline{k}_i=e^{-h}(k_i-N(h)) \text{ for } i=1,\dots,n-1,
\end{equation}
where $N(h)=\langle\nabla h,N\rangle$. As $h=u(|\vec{x}|^2)$, we obtain that
\[
\nabla h=2u'(|\vec{x}|^2)\vec{x},
\]
and we can write the equation \eqref{2} as
\[
\overline{k}_i=e^{-h}\left(k_i-2u'(|\vec{x}|^2)\langle \vec{x},N\rangle\right)\qquad i=1,\dots,n-1.
\]

If \eqref{1} is a parametrization of a minimal hypersurface in $(\R^n\setminus B_{R_0},\overline{g})$, then the above equation provides
\begin{equation}\label{3}
0=H-2(n-1)u'(|\vec{x}|^2)\langle \vec{x},N\rangle,
\end{equation}
where $H$ stands for the mean curvature at a point $\vec{x}\in \Sigma$.

By considering the expression of the normal vector $N$ and the position vector $\vec{x}$ above, we have
\begin{equation}\label{4}
\langle \vec{x},N\rangle=\frac{x'(t)t-x(t)}{\sqrt{1+x'(t)^2}}.
\end{equation}

Now, using the expression of $k_i$ and \eqref{4}, we can rewrite \eqref{3} as
\begin{equation}\label{5}
2(n-1)u'(|\vec{x}|^2)\frac{tx'(t)-x(t)}{\sqrt{1+x'(t)^2}}=-\frac{x''(t)}{\sqrt{(1+x'(t)^2)^3}}+\frac{n-2}{x(t)\sqrt{1+x'(t)^2}}.
\end{equation}

Hence, we obtain from \eqref{5} that $x(t)$ is a solution to the second order ODE
\begin{equation}\label{6}
x''(t)=\left[-2(n-1)u'(|\vec{x}|^2)(x'(t)t-x(t))+\frac{n-2}{x(t)}\right](1+x'(t)^2).
\end{equation}

Thus, there is a correspondence between rotationally symmetric free boundary minimal hypersurfaces in $(\R^n\setminus B_{R_0},\overline{g})$ parametrized by  \eqref{1} and solutions to the differential equation \eqref{6} given the initial conditions of interest.

For the sake of clarity, we develop all the computations supossing $\overline{g}$ is the Schwarzschild metric, i.e. we take $\overline{g}=g_{Sch(k)}$ with $k=1$. Then, we need to find a solution to the equation \eqref{6} with the initial condition $x(t_0)=\sqrt{R_0^2-t_0^2}$, $x'(t_0)t_0=x(t_0)$ where $t_0\in (0,R_0)$.  The equation \eqref{6} says that $\Sigma$ is a minimal hypersurface in the Schwarszchild space $(\R^n\setminus B_{R_0},\overline{g})$, and the conditions $x(t_0)=\sqrt{R_0^2-t_0^2}$, $x'(t_0)t_0=x(t_0)$ where $t_0\in (0,R_0)$ say that $\Sigma$ is free boundary (see Figure 1).
\begin{figure}[h]
\includegraphics[scale=0.4]{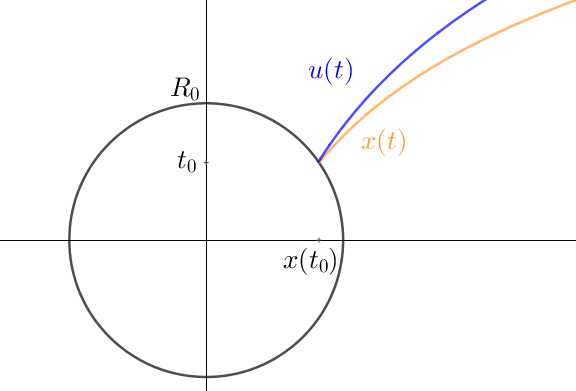}
\caption{}
\end{figure}

As $u(|\vec{x}|^2)=\frac{2}{n-2}\ln\left(1+\frac{m}{2|\vec{x}|^{n-2}}\right)$, we obtain
\[
u'(|\vec{x}|^2)=-\frac{m}{|\vec{x}|^2(2|\vec{x}|^{n-2}+m)}
\]
and, consequently,
\begin{equation}\label{6.25}
u'(|\vec{x}|^2)=-\frac{m}{2(x(t)^2+t^2)^{n}+m(x(t)^2+t^2)},
\end{equation}
for all $\vec{x}\in \Sigma$ (note that $|\vec{x}|^2=x(t)^2+t^2$, if $\vec{x}\in \Sigma$).

Now set
\[
\chi(x,t)=\frac{2m(n-1)}{2(x^2+t^2)^{n}+m(x^2+t^2)},
\]
then we can rewrite \eqref{6} as
\begin{equation}\label{main}
 x''(t)=\left[(x'(t)t-x(t))\chi(x(t),t)+\frac{n-2}{x(t)}\right](1+x'(t)^2).
\end{equation}

\begin{lemma}\label{lema1bis}
Let $\chi$ as above, then the solutions of \eqref{main} verifies that $x(t)$ goes to infinity as $t$ goes to $t^*$,
where $t^*\in \R\cup \{+\infty\}$ is chosen such that $(t_0,t^*)$ is the  maximal interval of existence of $x$. 
\end{lemma}
\begin{proof}
First, define $\Psi(t):=t x'(t)-x(t)$, we note that the initial conditions of $x$ are equivalent to $\Psi(t_0)=0$. Taking derivatives, we have
\begin{equation}\label{6.5}
\begin{split}
\Psi'(t)&=t x''(t)=t\left(\Psi(t)\chi(x(t),t)+\frac{n-2}{x(t)}\right)(1+x'(t)^2)\\
&>t\Psi(t)\chi(x(t),t)(1+x'(t)^2)\geq 0.
\end{split}
\end{equation}

Thus, $\Psi'>0$ and $\Psi\geq 0$, and we get $x'(t)\geq x(t)/t>0$. Consequently, $x$ is strictly increasing and can not happend that $x(t)<0$ for certain $t\in (t_0,t^*)$. Therefore, either the limit of $x$ is a constant as $t$ approaches $t^*$ or $x$ diverges when $t$ reaches $t^*$.

If $t^*<+\infty$ then since $x,x'>0$ if follows that $x$ diverges as $t$ approaches $t^*$. In the case of $t^*=+\infty$ we have that the limit of $x$ can not be a constant. If it was, by using \eqref{main} we would have $x''\geq\varepsilon$ for certain $\varepsilon>0$. In particular the limit of $x'$ would aproach infinity as $t$ goes to infinity which is a contradiction. Therefore in this second case $x$ also diverges.
\end{proof}

\begin{lemma}\label{lemma2}
For $n=3$ the solution of \eqref{main} is defined in $[t_0,+\infty)$.
\end{lemma}
\begin{proof}
For $n=3$ equation \eqref{main} is written as
\begin{equation}\label{cota}
\begin{split}
x''(t)&=\left[(t x'(t)-x)\frac{4 m}{2(x(t)^2+t^2)^3+m(x(t)^2+t^2)}+\frac{1}{x(t)}\right](1+x'(t)^2)\\
&\leq \left[(t x'(t))\frac{4 m}{2(x(t)^2+t^2)^3}+\frac{1}{x(t)}\right](1+x'(t)^2).
\end{split}
\end{equation}
Suppose $t<1$, then \eqref{cota} can by bounded by
\begin{equation*}
\begin{split}
x''(t)&\leq \left[x'(t)\frac{4 m}{2\left(x(t)^2+t_0^2\right)^3}+\frac{1}{x(t)}\right](1+x'(t)^2)\\
&=\left[x'(t)\frac{4 m}{2\left(t_0^6+3t_0^4x(t)^2+3t_0^2 x(t)^4+x(t)^6\right)}+\frac{1}{x(t)}\right](1+x'(t)^2)\\
&\leq \left[x'(t)\frac{2 m}{3 t_0^2 x(t)^4}+\frac{1}{x(t)}\right](1+x'(t)^2).
\end{split}
\end{equation*}
If $t\geq 1$ then 
\begin{equation*}
\begin{split}
x''(t)&\leq \left[(t x'(t))\frac{4 m}{2\left(t^6+3t^4x(t)^2+3t^2 x(t)^4+x(t)^6\right)}+\frac{1}{x(t)}\right](1+x'(t)^2)\\
&\leq \left[x'(t)\frac{2 m}{3x(t)^4}+\frac{1}{x(t)}\right](1+x'(t)^2)
\end{split}
\end{equation*}
In conclusion, \eqref{cota} can be bounded by an autonomuos expression:
\begin{equation}\label{autonomuos}
x''(t)\leq\left[x'(t)\frac{a}{x(t)^4}+\frac{1}{x(t)}\right](1+x'(t)^2)
\end{equation}
where $a=\max\left\lbrace \frac{2m}{3},\frac{2m}{3t_0^2}\right\rbrace>0$. Set $x(t)=e^{u(t)}$, we can rewrite \eqref{autonomuos} as
\[
e^{u(t)}u''(t)\leq \frac{a u'(t)}{e^{3u(t)}}+\frac{a u'(t)^3}{e^{u(t)}}+\frac{1}{e^{u(t)}}.
\]
If we study the equality in the last expression; since, by Lemma \ref{lema1bis}, $u(t)$ goes to infinity; it yields with a function $v$ (imposing the same initial conditions of $u$), with the same asymptotic behavior as $u$ , verifying
\begin{equation}
e^{v(t)}v''(t)=\frac{a v'(t)}{e^{3v(t)}}+\frac{a v'(t)^3}{e^{v(t)}}.
\end{equation}
Now make the following estimate:
\begin{equation}\label{7}
v''(t)\leq\frac{a (v'(t)+v'(t)^3)}{e^{2v(t)}}\leq\frac{a (v'(t)+v'(t)^3)}{e^{v(t)}}.
\end{equation}
Again, studying the equality in \eqref{7} will give an upper bound $\widetilde{v}$ for $v$. If we set $w(\widetilde{v})=\frac{d \widetilde{v}}{d t}$ then 
\[
\frac{d^2 \widetilde{v}(t)}{dt^2}=w(\widetilde{v})\frac{dw(\widetilde{v})}{d\widetilde{v}},
\]
and we can rewrite the equality in \eqref{7} as
\[
w \dot w=\frac{a (w^3+w)}{e^{\widetilde{v}}},
\]
where $\cdot=\frac{d}{d \widetilde{v}}$. Solving this equation gives that either $w$ is constant or $w(\widetilde{v})=-\tan\left(a e^{-\widetilde{v}}+c_1\right)$, for some constant $c_1\in \R$. Thus
\begin{equation}\label{8}
\widetilde{v}'(t)=-\tan(a e^{-\widetilde{v}(t)}+c_1).
\end{equation}
Indeed, we can compute $c_1$. Since the initial contitions of $x$ are $x(t_0)=\sqrt{R_0^2-t_0^2}$, $t_0 x'(t_0)=x(t_0)$, the initial conditions for $u$ given by $x(t)=e^{u(t)}$ (and also for $v$, and $\widetilde{v}$) are $u(t_0)=\ln\left(\sqrt{R_0^2-t_0^2}\right)$ and $u'(t_0)=1/t_0$. Plugging this conditions on \eqref{8} gives 
\begin{equation}\label{8bis}
c_1=-\arctan\left(\frac{1}{t_0}\right)-\frac{a}{\sqrt{R_0^2-t_0^2}}.
\end{equation}
If $[t_0,t^*)$ is the maximal interval of definition of $v$, when $t$ goes to $t^*$ then $\widetilde{v}(t)$ approaches infinity. By \eqref{8}, if $c_1\not=\pi/2+k\pi$ for any $k\in \mathbb{Z}$, we can also conclude that $\widetilde{v}'(t)$ approaches a constant value as $t$ goes to $t^*$. If there exists $k\in \mathbb{Z}$ such that $c_1=\pi/2+k \pi$, then it is enough to consider a perturbation $\widetilde{a}>a$ of the parameter $a$ to bound equation \eqref{autonomuos} which makes $c_1\not=\pi/2+k \pi$ for any $k\in \mathbb{Z}$. This can be done because $c_1$ varies continuosly on $a$. In both cases $\widetilde{v}'(t)$ approaches a constant value as $t$ goes to $t^*$. This behaviour is not possible if $t^*<+\infty$, so we conclude that $t^*=\infty$ and $x$, which has the same asymptotic behavior as $\widetilde{v}$, is defined in $[t_0,+\infty)$.
\end{proof}

\begin{theorem}\label{2.3}
For each $t_0\in (0,R_0)$ there exists a complete rotationally symmetric free boundary minimal hypersurface $\Sigma_{t_0}$ in the Schwarszchild space $(\R^n\setminus B_{R_0},\overline{g})$, $n\geq 3$ that intersects $\partial\R^n\setminus B_{R_0}$ at height $x_n=t_0$.

Moreover, {if $n=3$ then the height of $\Sigma_{t_0}$ with respect to the $x_n=0$ plane is unbounded  and} if $n\geq 4$ then $\Sigma_{t_0}$ is bounded above by the hyperplane of height $h_0$, where
\[
h_0=t_0+\int_{\sqrt{R_0^2-t_0^2}}^{+\infty}\frac{d\mu}{\sqrt{\left(\frac{R_0}{t_0}\right)^2\left(\frac{\mu}{\sqrt{R_0^2-t_0^2}}\right)^{2(n-2)}-1}}.
\]
\end{theorem}
\begin{proof}
It follows from proof of Lemma \ref{lema1bis} that the function $x$ is strictly increasing. Thus, either $x$ is defined for all $t\geq t_0$ or $x$ blows up in finite time. Since $X$ is a parametrization of $\Sigma_{t_0}$ using cylindrical coordinates, if $x$ is defined for all $t\geq t_0$ then $\Sigma_{t_0}$ is complete and has unbounded height with respect to the $x_n=0$ plane and if $x$ blows up in finite time, then $\Sigma_{t_0}$ converges asymptotically to a plane of certain height. In both cases $\Sigma_{t_0}$ is complete, hence, we get existence.

Using \eqref{main}, we get 
\[
x''(t)\geq \frac{n-2}{x(t)}(1+x'(t)^2).
\]
Therefore, the solution $u$ to the ODE
\begin{equation}\label{main2}
	\begin{split}
	u''(t)=\frac{n-2}{u(t)}(1+u'(t)^2),
	\end{split}
\end{equation}
with initial conditions $u(t_0)=\sqrt{R_0^2-t_0^2}$, $u'(t_0)t_0=u(t_0)$ is a lower bound in the $(t,u)$-plane for the solution of \eqref{main} (see Figure 1). Multiplying by $u'(t)$ we can rewrite \eqref{main2} as
\[
\frac{2 u'(t)u(t)}{1+u'(t)^2}=2(n-2)\frac{u'(t)}{u(t)},
\]
which can be integrated:
\[
\log(1+u'(t)^2)=2(n-2)\log(u(t))+c,
\]
for some $c\in \R$ depending on the initial conditions. The above equation is equivalent to
\[
u'(t)=\sqrt{e^c u(t)^{2(n-2)}-1},
\]
which can be integrated in $u$, getting
\[
t(u)=t_0+\int_{u(t_0)}^u\frac{d\mu}{\sqrt{e^c \mu^{2(n-2)}-1}}.
\]
Now impose $u'(t_0)t_0=u(t_0)$ to get
\begin{equation}\label{9}
t(u)=t_0+\int_{\sqrt{R_0^2-t_0^2}}^{u}\frac{d\mu}{\sqrt{\left(\frac{R_0}{t_0}\right)^2\left(\frac{\mu}{\sqrt{R_0^2-t_0^2}}\right)^{2(n-2)}-1}}.
\end{equation}

For $n=3$, it follows from Lemma \ref{lemma2} that the height of $\beta$ given by $\beta(t)=(x(t),0,\dots,0,t)$, $t\in [t_0,\infty)$ with respect to the plane $x_3=0$ is unbounded. 

For $n\geq 4$ 
\[
\lim_{u\to +\infty} t(u)=h_0<+\infty,
\]
so $u$ blows up in finite time. Since $u$ is a lower bound for $x$ then $x$ also blows up and the height of $\Sigma_{t_0}$ is bounded above by $h_0$.
\end{proof}

Moreover, for $n=3$, using the notation of the previous proof, since 
\[
\lim_{u\to +\infty} t(u)=+\infty,
\]
we have a control on the growth of the solutions of \eqref{main}.

\begin{remark}\label{remark1}
The hypersurfaces obtained in Theorem \ref{2.3} are not totally geodesic. By contradiction, if there exists $t_0\in (0,R_0)$ such that $\Sigma_{t_0}$ is totally geodesic, then since totally geodesic surfaces are preserved by conformal changes of the metric we deduced that $\Sigma_{t_0}$ is a flat plane. In fact, non-compact 
totally geodesic surfaces in the Schwarzschild space are flat planes. Moreover since we are asking $\Sigma_{t_0}$ to be rotationally symmetric then it must be an horizontal plane that intersects $\partial M^n$ at height $x_n=t_0$. However, these horizontal planes are not free boundary planes, which is a contradiction. 
\end{remark}

\begin{remark}
Theorem \ref{2.3} together with Remark \ref{remark1} directly proves Theorem \ref{MAIN}. For the slightly general situation of Theorem \ref{newMAIN2} (which, in particular includes Theorem \ref{newMAIN}) the fact $u'(|\vec{x}|^2)<0$ is used to generalise the inequality \eqref{6.5} for the slightly general situation of \eqref{6}: if $\Psi(t)=t x'(t)-x(t)$ then
\begin{equation*}
\begin{split}
\Psi'(t)&=t x''(t)=t\left(\Psi(t)(-2(n-1)u'(|\vec{x}|^2))+\frac{n-2}{x(t)}\right)(1+x'(t)^2)\\
&>t\Psi(t)\chi(x(t),t)(1+x'(t)^2)\geq 0,
\end{split}
\end{equation*}

The other restriction ($-u'(|\vec{x}|^2)\leq \frac{1}{|\vec{x}|^\alpha}$ for $\alpha\geq 2$) is needed to prove the anologous result of Lemma \ref{lemma2} for equation \eqref{6}.
\end{remark}

\begin{lemma}\label{lemma2bis}
For $n=3$ the solution of \eqref{6} is defined in $[t_0,+\infty)$.
\end{lemma}
\begin{proof}
For $n=3$ equation \eqref{6} is written as
\begin{equation}\label{cotabis}
\begin{split}
x''(t)&=\left[-4(t x'(t)-x)u'(|\vec{x}|^2)+\frac{1}{x(t)}\right](1+x'(t)^2)\\
&\leq \left[(t x'(t))\frac{4}{(x(t)^2+t^2)^\alpha}+\frac{1}{x(t)}\right](1+x'(t)^2).
\end{split}
\end{equation}
In order to follow the proof of Lemma \ref{lemma2} we need de following estimate, for $a,b\geq 0$, $\alpha> 0$ and $\beta\in (0,\alpha)$:
\begin{equation}\label{estimate}
(a+b)^\alpha\geq a^\beta b^{\alpha-\beta}\Rightarrow\left(\frac{1}{(a+b)^\alpha}\leq \frac{1}{a^\beta b^{\alpha-\beta}}\right)
\end{equation}
that we are going to apply for $\alpha\geq 2$ and $\beta=3/2$.
Suppose $t<1$, then \eqref{cotabis} can by bounded by
\begin{equation*}
x''(t)\leq \left[x'(t)\frac{4}{x(t)^3 t_0^{2\alpha-3}}+\frac{1}{x(t)}\right](1+x'(t)^2)
\end{equation*}
If $t\geq 1$ then 
\begin{equation*}
\begin{split}
x''(t)&\leq \left[(x'(t))\frac{4}{x(t)^3 t^{2\alpha-4}}+\frac{1}{x(t)}\right](1+x'(t)^2)\\
&\leq \left[(x'(t))\frac{4}{x(t)^3}+\frac{1}{x(t)}\right](1+x'(t)^2)
\end{split}
\end{equation*}
In conclusion, \eqref{cotabis} can be bounded by an autonomuos expression:
\begin{equation}\label{autonomuosbis}
x''(t)\leq\left[x'(t)\frac{a}{x(t)^3}+\frac{1}{x(t)}\right](1+x'(t)^2)
\end{equation}
where $a=\max\left\lbrace \frac{4}{t_0^{2\alpha-3}},4\right\rbrace>0$. Set $x(t)=e^{u(t)}$, we can rewrite \eqref{autonomuosbis} as
\[
e^{u(t)}u''(t)\leq \frac{a u'(t)}{e^{2u(t)}}+a u'(t)^3+\frac{1}{e^{u(t)}}.
\]
Again, we study the equality in the last expression. Since $u(t)$ goes to infinity; we found a function $v$, with the same asymptotic behavior as $u$ , verifying
\begin{equation}
e^{v(t)}v''(t)=\frac{a v'(t)}{e^{2v(t)}}+a v'(t)^3.
\end{equation}
Now make the following estimate:
\begin{equation}\label{7bis}
v''(t)\leq\frac{a (v'(t)+v'(t)^3)}{e^{v(t)}}.
\end{equation}
The rest of the proof matches exactly with the one given in Lemma \ref{2}.
\end{proof}

\section{Morse Index and total curvature}

For each $t_0\in (0,R_0)$, consider the proper non totally geodesic rotationally symmetric free boundary minimal surface $\Sigma_{t_0}$ in the Schwarszchild space $M^3_1$ constructed in the Theorem \ref{MAIN}.  It follows from the work of A. Carlotto \cite{C16} that the Morse index of $\Sigma_{t_0}$ can not be zero. Moreover, the Morse index of $\Sigma_0$ is one, as showed by R. Montezuma \cite{RaMo}. However, the Morse index of the surface $\Sigma_0$ was not computed in $M^3_k$ for $k>1$. In this section, we compute the index of $\Sigma_0$ in $M^3_k$ for $k\in [1,6/5]$ and for any $k\geq 1$ if the parameter $m$ is big enough. Since $\Sigma_{t_0}$ converges to $\Sigma_0$ as $t_0$ goes to 0 then the Morse index of $\Sigma_{t_0}$ is also one, at least for $t_0$ very close to 0.

Consider the quadratic form $Q_{\Sigma}(\cdot, \cdot)$ for an hypersurface $\Sigma\subset (M^n,g_{Sch(k)})$ given by
\begin{equation}\label{form1}
\begin{split}
Q_{\Sigma}(u,u) & = - \int_{\Sigma}u\left( \Delta_{({\Sigma},g)}u + ({{\rm Ric}_g}(N ,N) +|A_{\Sigma}|^2)u \right)dv_{\Sigma} \\ 
 & \qquad \qquad +\int_{\partial \Sigma}u\left(\frac{\partial u}{\partial \eta} - A_{\partial M}(N ,N )u \right)ds\,,
\end{split}
\end{equation}
where ${{\rm Ric}}_g$ denotes the Ricci  curvature of $(M^{n}, g_{Sch(k)})$ and  $A_{\Sigma}$, $A_{\partial M}$ stand for the second fundamental forms of $\Sigma$ and $\partial M$:
\[
A_{\partial M}(V,W) = -g_{Sch(k)}(\bar{\nabla} _{\eta}V,W),\quad A_{\Sigma}(V,W) = -g_{Sch(k)}(\bar{\nabla}_N V,W)\,.
\] 
where $\eta$ is the inwards pointing unit normal direction in $\partial M$ and $N$ denotes a unit normal vector field along $\Sigma$.

Let $R>R_0$ and define $\Sigma(R) :=\Sigma \cap \{|\vec{x}| \leq R\}$. The Morse index  ${\rm Ind}_{F}(\Sigma(R))$ of $\Sigma(R)$ is defined as the maximal dimension of a linear subspace $V$ of smooth functions $u:\Sigma(R)\rightarrow \mathbb{R}$ vanishing on $\Sigma\cap \{|\vec{x}| = R\}$ such that $Q_{\Sigma}(u, u)<0$, for all $u \in V\setminus \{0\}$. 

Taking limits, we define the Morse index ${\rm Ind}(\Sigma)$ of $\Sigma$ as  
\[
{\rm Ind}(\Sigma) := \limsup\limits_{R\rightarrow +\infty}{\rm Ind}_{F}(\Sigma(R)),
\]
which can be infinite.

Additionally, we can compute the Morse index ${\rm Ind}_F(\Sigma(R))$ as the number of negative eigenvalues, counting multiplicities, of the problem 
\[
(P) \quad\left\{ \begin{array}{ccc}
      \Delta_{({\Sigma},g)} \psi + ({{\rm Ric}_g}(N,N) +|A_{{\Sigma}}|^2)\psi =-\lambda \psi & \text{ in } & \Sigma(R) \\
     \psi=0 &\text{ on } &S(R)\cap \Sigma \\
     \dfrac{\partial \psi}{\partial \eta}=0  & \text{ on } & S_{0}\cap\partial\Sigma
    \end{array} \right.  
\]
where $S(R)=\mathbb{S}^{n-1}(R)$ is the sphere centered at the origin and radius $R$ and $S_0=S(R_0)$. 

From now on, we will denote $g:=g_{Sch(k)}$.

\begin{theorem}\label{MAINindex}
If $k\in (1,6/5]$ then the Morse index of $\Sigma_0\subset M^3_k$ is one. Moreover if we choose $m=m(k)$ big enough to ensure
\[
\frac{4 m
   r^{\frac{3}{k}+2}}{\left(2
   r^{3/k}+m r^2\right)^2}\leq \frac{3}{16}, \quad r\geq R_0
\]
then for all $k> 1$, $k\not=\frac{3}{2}$, the Morse index of $\Sigma_0\subset M_k^3$ is one.

Under this conditions, there exists a constant $0<L_0\leq R_0$ such that for all $t_0\in (0,L_0)$, the Morse index of $\Sigma_{t_0}$ is 1.
\end{theorem}

\begin{proof}
From the Gauss equation, it follows that
\[
\frac{R_g}{2}-{\rm Ric}_g(N,N)-|A_{g}|^2=K_{g}-\frac{|A_{g}|^2}{2}
\]
\[
=-e^{-2f}(\Delta_{(\Sigma,\langle,\rangle)} f-K)-\frac{1}{2}(|A|^2-\frac{H^2}{2})e^{-2f}\,
\]
where $R_g$ denotes the scalar curvature of $(M^3_k,g)$ and $K$, $A$ are the Gaussian curvature  and the second fundamental form of $\Sigma$ computed with respect to the euclidean metric (or with respect to the metric $g$ if it appears as a subindex). In the previous expression we have used that if  the dimension of $\Sigma$ is $2$, and the metric writes as $g=e^{2f}\left\langle,\right\rangle$ then $\Delta_{(\Sigma,\langle,\rangle)} f-K=-K_{g}e^{2f}$. Moreover, the Laplacian is conformal in dimension $2$: 
\[
\Delta_{(\Sigma,g)}=e^{-2f}\Delta_{(\Sigma,\langle,\rangle)}\,.
\]
Using the formula given in (\ref{2}), we also obtain $|A_{g}|^2=e^{-2f}(|A|^2-\frac{1}{2}H^2)$. 

Let $\Sigma : = \Sigma_{t_0}$ be the surface constructed in the previous sections. The Schwarzschild metric is conformal to the Euclidean metric by $g_{Sch(k)}=e^{2f}\langle, \rangle$, where
\begin{equation}
f( r=|\vec{x}| )=\frac{2k}{3-2k}\ln \left( 1+\frac{m}{2|\vec{x}|^{\frac{3}{k}-2}} \right), \quad \vec{x}\in M_k^3\,. 
\end{equation}
Hence, using the identities above and the Gauss equation $2K=H^2-|A|^2$, we obtain
\begin{eqnarray*}
J_{\Sigma_{t_0}}u&:=& \Delta_{({\Sigma},g)} u + ({\rm Ric}_g(N,N) +|A_{g}|^2)u \\
&=&e^{-2f}\left(\Delta_{(\Sigma,\langle,\rangle)}u+(\frac{1}{2}e^{2f}R_g+\Delta_{(\Sigma,\langle,\rangle)}f)u +( |A|^2-\frac{3}{4}H^2)u\right)\,.
\end{eqnarray*}
First, we compute  the escalar curvature $R_g$ of $M_k^3$:
\begin{equation}\label{escalar}
R_g=-8e^{-2f}\left(v^{-1}\Delta_{(M_k^3,\langle,\rangle)}v\right)\,,
\end{equation}
where 
\[
e^{2f}=v^{4}\,.
\]
The computations in \eqref{escalar} give us:
\[
R_g=\frac{48 (k-1) m r^{3/k}}{k
   \left(2 r^{3/k}+m
   r^2\right)^2}e^{-2f}\,.
\]
Now, we obtain $\Delta_{(\Sigma,\langle,\rangle)}f$ for the surface $\Sigma_0$:
\[
\Delta_{(\Sigma,\langle,\rangle)}f=\left(\frac{3-2k}{k}\right)\frac{m}{r^{\frac{3}{k}}}\left( 1+\frac{m}{2r^{\frac{3}{k}-2}} \right)^{-2}.
\]   
Here we have used that $\Delta_{(\Sigma,\langle,\rangle)}f=f''+\frac{1}{r}f'$.

Note that, on $\Sigma_0$, we have $|A|^2-\frac{3}{4}H^2=0$. Thus, $\psi$ is a solution of
\[
\left\{ \begin{array}{ccc}
      \Delta_{({\Sigma},g)} \psi + ({\rm Ric}_g(N,N) +|A_{g}|^2)\psi =-\lambda \psi & \text{ in } & \Sigma(R) ,\\
     \psi=0 &\text{ on } &S(R)\cap \partial\Sigma(R) ,\\
     \dfrac{\partial \psi}{\partial \eta}=0 & \text{ on } & S_{0}\cap\partial\Sigma(R) ,
    \end{array} \right.  
\]
if, and only if, $\psi$ satisfies

\[
(R)\quad \left\{ \begin{array}{ccc}
      \Delta_{(\Sigma,\langle,\rangle)}\psi+Q \psi =-\lambda \psi \left( 1+\frac{m}{2r^{\frac{3}{k}-2}} \right)^{\frac{4k}{3-2k}} & \text{ in } & \Sigma(R) ,\\
     \psi=0 &\text{ on } &S(R)\cap \partial\Sigma(R) ,\\
     \dfrac{\partial \psi}{\partial \eta}=0 & \text{ on } & S_{0}\cap\partial\Sigma(R) ,
    \end{array} \right.  
\]
where 
\begin{eqnarray*}
Q&=& \left(\frac{3-2k}{k}\right)\frac{m}{r^{\frac{3}{k}}}\left( 1+\frac{m}{2r^{\frac{3}{k}-2}} \right)^{-2}+\frac{24 (k-1) m r^{3/k}}{k\left(2 r^{3/k}+mr^2\right)^2}\\
&=&\frac{4 (4 k-3) m r^{3/k}}{k
   \left(2 r^{3/k}+m
   r^2\right)^2}\,.
\end{eqnarray*}
Now it is a standard result that every solution of $(R)$ has the form
\[
v(r,\theta)=\sum\limits_{l=0}^{+\infty}a_l(r)\Phi_l(\theta)\,,\quad \theta\in \mathbb{S}^{1}\, ,\quad r\geq0,
\]
where $\Phi_0$ is constant, and for $l\geq1$, $\Phi_l$ is an eigenfunction of the Laplacian on the unit circle $\mathbb{S}^{1}$.

Hence, the function $a_l$ satisfies the equation
\[
\left\{ \begin{array}{ccc}
     L_l(a_l) = 0 & \text{ in } & \Sigma(R) ,\\
      a'_l(R_0)=0 \,,\\
     a_l(R)=0 \,,
    \end{array} \right.  
\]
where
\[
L_l:=\frac{d^2}{dr^2}+\frac{1}{r}\frac{d}{dr} +\frac{\lambda_l}{r^2} + Q
+\lambda \left( 1+\frac{m}{2r^{\frac{3}{k}-2}} \right)^{\frac{4k}{3-2k}} \,,
\]
and $\lambda_l=-l^2$, $l\geq0$.

By making the change of variables $v_l(r)=\sqrt{r}\cdot a_l(r)$, we obtain that
\[
\left\{ \begin{array}{ccc}
     \frac{d^2v_l}{dr^2} +A_{\lambda}\cdot v_l= 0 & \text{ in } & [R_0,R] ,\\
      v'_l(R_0)=\frac{1}{2R_0}v_l(R_0) \,,\\
     v_l(R)=0 \,,
    \end{array} \right.  
\]
where 
\[
A_{\lambda}=\frac{1}{4r^2}-\frac{l^2}{r^2} + Q
+\lambda \left( 1+\frac{m}{2r^{\frac{3}{k}-2}} \right)^{\frac{4k}{3-2k}} \,.
\]
Integrating by parts, we obtain 
\begin{equation}\label{parts}
\int\limits_{R_0}^{R}(v_l')^2dr+\frac{1}{2R_0}v_l^2(R_0)=\int\limits_{R_0}^{R}v^2_l\cdot A_{\lambda}dr\,.
\end{equation}
In particular if $A_\lambda\leq 0$ for some $\lambda\in \R$ then this expression ensures that the corresponding $v_l$ vanish. 

Using that 
\[
\frac{4k-3}{k}\leq \frac{3}{2}, \quad k\in (1,6/5], 
\]
\[
\frac{4 m r^{3/k}}{\left(2
   r^{3/k}+m r^2\right)^2}=\frac{4 m
   r^{\frac{3}{k}+2}}{\left(2
   r^{3/k}+m r^2\right)^2}\frac{1}{r^2}\overset{(*)}{\leq} \frac{1}{2r^2},
\]
where $(*)$ holds because 
\[
\frac{2ab}{(a+b)^2}\leq\frac{1}{2}, \quad a,b\in \R,
\]
we have the following bound for $Q$:
\begin{equation}\label{bound}
Q\leq \frac{3}{4r^2}.
\end{equation}
Since 
\[
\frac{4 m
   r^{\frac{3}{k}+2}}{\left(2
   r^{3/k}+m r^2\right)^2}\overset{m\to\infty}{\longrightarrow} 0
\]
uniformly in $r\geq R_0$ and 
\[
\frac{4k-3}{k}\leq 4, \quad k>1,
\]
the bound \eqref{bound} can be also obtained if $k>1$ and $m$ is big enough to ensure 
\begin{equation}\label{condicion}
\frac{4 m
   r^{\frac{3}{k}+2}}{\left(2
   r^{3/k}+m r^2\right)^2}\leq \frac{3}{16}, \quad r\geq R_0.
\end{equation}
Hence,
\begin{eqnarray*}
A_{\lambda}&=&\frac{1}{4r^2}-\frac{l^2}{r^2} + Q
+\lambda \left( 1+\frac{m}{2r^{\frac{3}{k}-2}} \right)^{\frac{4k}{3-2k}}\\
&\leq&\frac{1}{r^2}\left(1-l^2  \right)+\lambda \left( 1+\frac{m}{2r^{\frac{3}{k}-2}} \right)^{\frac{4k}{3-2k}} \,.
\end{eqnarray*}

Note that,  if $\lambda<0$, $l\geq 1$ then  $A_{\lambda}\leq 0$, for all $1< k\leq \frac{6}{5}$ (or $k>1$ and $m$ big enough to ensure \eqref{condicion}). Hence, by \eqref{parts}, $a_l=0$, for all $l\geq1$, if $\lambda<0$ and  $1< k\leq \frac{6}{5}$ (or $k>1$ and $m$ big enough to ensure \eqref{condicion}). This means that, with $\lambda<0$, the solutions of the problem $(R)$ are radial functions $a=a(r)$. Moreover, negative eigenvalues $\lambda$ associated with the eigenvalue problem (R) have multiplicity one (\cite{RaMo}, Corolary 2.3). In order to proof that ${\rm Ind}(\Sigma(R))$ is 0 or 1 we just need to check that the solution of (R) is positive inside $\Sigma(R)$. 

Now, we make an analysis as in \cite{RaMo}. We claim that the zeroes of $v_l$ are isolated whenever $v_l$ is not identically zero. In order to prove this last claim, by uniqueness of solutions, it suffices to verify that if $v_l(r_0)=0$, then $\frac{d}{dr}v_l(r_0)\neq0$. Hence, away from the zeroes of $v_l$, we define
\[
\gamma_l(r)=\frac{v'_l(r)}{v_l(r)}\,.
\]
Using the equation satisfied by $v_l$, we obtain
\[
\gamma'_l+(\gamma_l)^2=-A_{\lambda}
\]
with boundary condition $\gamma_l(R_0)=\frac{1}{2R_0}$. Again, note that
\[
Q\leq \frac{3}{4r^2}\, ,
\]
for all $r\geq R_0$.  Consequently, if $\lambda\leq0$,
\begin{eqnarray*}
\gamma'_l+(\gamma_l)^2&\geq&-\frac{1}{r^2}\left(1-l^2  \right)-\lambda \left( 1+\frac{m}{2r^{\frac{3}{k}-2}} \right)^{\frac{4k}{3-2k}}\\
&\geq& \frac{1}{r^2}\left( \frac{\alpha_0^2}{4}-\frac{1}{4} \right)\,,
\end{eqnarray*} 
where $\alpha^2_0=4l^2-3$. Now, we observe that the function
\[
f(r)=\frac{1}{2r}\left(1-\alpha\left( \frac{2}{1+\left(\frac{2r}{m} \right)^{\alpha}}-1 \right)  \right)\,,
\]
for a non-negative constant $\alpha$, satisfies $f(\frac{m}{2})=\frac{1}{m}$ and
\[
f'(r)+f^2(r)=\frac{1}{r^2}\left( \frac{\alpha^2}{4}-\frac{1}{4} \right)\,.
\]
Now, the function
\[
g(r)=a_0f(a_0r)\,,
\]
where $a_0=\frac{m}{2R_0}$, satisfies $g(R_0)=\frac{1}{2R_0}$ and
\[
g'(r)+g^2(r)=\frac{1}{r^2}\left( \frac{\alpha^2}{4}-\frac{1}{4} \right)\,.
\]
Therefore, $\gamma_l(r)\geq g(r) $. Now, by contradiction, if $v_l$ changes sign then there exists $R_0<r_0<R$ such that $v_l(r_0)=0$. Since  $\gamma_l(r)\rightarrow-\infty$ as $r\rightarrow r^-_0$, and $g(r)$ does not satisfies this "explosion" when $r\rightarrow r_0^-$ we reach a contradiction that allows us to conclude that for $\lambda\leq0$, $v_l$ does not change sign.

Since the only eigenfunction that does not change its sign is that one associated with the first eigenvalue, this implies that the Morse index of $\Sigma(R)$ is at most one. This gives us that either ${\rm Ind}(\Sigma(R))=0$ or ${\rm Ind}(\Sigma(R))=1$. 

To prove that ${\rm Ind}(\Sigma_0)=1$, we first note that $R_g>0$ if $k>1$ (equation \eqref{escalar}). If ${\rm Ind}(\Sigma_{0})=0$, it follows from Hong-Saturnino's Theorem (see Corollary 4.5 in \cite{HS}) that $\Sigma_0$ is totally geodesic and $R_g=0$ on $\Sigma_0$ which is a contradiction. Therefore, ${\rm Ind}(\Sigma_{0})=1$.
\end{proof}

\subsection{Total Curvature}

In this subsection, we are interested in verifying that the surfaces $\Sigma_{t_0}\subset M^3$, $t_0\in (0,R_0)$ have finite total curvature. In order to do that, we will use the following facts:
\begin{itemize}
\item[(1)] $\frac{x''}{1+(x')^2}=\frac{d}{dt} \arctan(x')$;
\item[(2)] the function $\arctan$ is bounded;
\item[(3)] the volume element $dv_{\delta}$ of $(\Sigma_{t_0},\langle,\rangle)$ is given by $dv_{\delta}=x\sqrt{1+(x')^2}dtd\theta$;
\item[(4)] there exists $t^*\geq 0$ such that for any $t> t^*$ then $x(t)>1$. Therefore $\frac{1}{x^4}<\frac{1}{x}$, for all $t> t^*$.
\end{itemize}
First, it follows from $(2.12)$ that
\[
x''\leq \left( x'\frac{a}{x^4}+\frac{1}{x} \right)(1+(x')^2)\,.
\]
Hence,
\begin{eqnarray*}
k_1^2&=&\frac{(x'')^2}{(1+(x')^2)^3}\\
&\leq& x'x''\frac{a}{x^4(1+(x')^2)^2}+\frac{x''}{x(1+(x')^2)^2}\\
&=& x'x''\frac{a}{x^4(1+(x')^2)^2}\frac{\sqrt{1+(x')^2}}{\sqrt{1+(x')^2}}+\frac{x''}{x(1+(x')^2)^2}\\
&\leq& x''\frac{a}{x^4(1+(x')^2)^2}(\sqrt{1+(x')^2})+\frac{x''}{x(1+(x')^2)^2}\,,
\end{eqnarray*}
where we have used that 
\[
\frac{x'}{\sqrt{1+(x')^2}}\leq 1\,.
\]
Therefore, using the facts (1), (2), (3) and (4), we obtain that
\[
\int_{\Sigma_{t_0}}k_1^2dv_{\delta}<+\infty\,.
\]
Also, we get that
\[
\int_{\Sigma_{t_0}}k_2^2dv_{\delta}<+\infty\,,
\]
since
\[
x''\geq \frac{1+(x')^2}{x} 
\]
implies that
\[
k_2^2= \frac{1}{(1+(x')^2)x^2}\leq\frac{x''}{x(1+(x')^2)^2}\,. 
\]
We conclude then
\[
\int_{\Sigma_{t_0}}|A|^2dv_{\delta}<+\infty\,,
\]
which give us that the total curvature of $(\Sigma_{t_0},g_{Sch(1)}=\overline{g})$ is finite:

\[
\int_{\Sigma_{t_0}}|A_{\overline{g}}|^2dv_{\overline{g}}=\int_{\Sigma_{t_0}}\left(|A|^2-\frac{H^2}{2}\right)dv_{\delta}<+\infty\,.
\]
The same proof holds for $\Sigma_{t_0}\subset M^3_k$ if we use equation \eqref{autonomuosbis} instead of \eqref{autonomuos}.

\bibliographystyle{amsplain}

\end{document}